\documentclass[10pt]{amsart}

\usepackage{amsmath}
\usepackage{amsfonts}
\usepackage{amssymb}
\usepackage{amsthm}
\usepackage{url}
\usepackage{tikz-cd}
\usepackage{dsfont}
\usepackage{graphicx}
\usepackage{caption}
\usepackage{subcaption}
\usepackage{comment}
\usepackage[colorlinks,pagebackref,hypertexnames=false]{hyperref}
\usepackage{color}
\usepackage{enumerate}
\usepackage{mathrsfs}

\usepackage{todonotes}

\numberwithin{equation}{section}

\newtheorem{proposition}{Proposition}[section]
\newtheorem{lemma}[proposition]{Lemma}
\newtheorem{theorem}[proposition]{Theorem}

\theoremstyle{definition}
\newtheorem{remark}[proposition]{Remark}
\newtheorem{definition}[proposition]{Definition}

\DeclareMathOperator{\End}{End}
\DeclareMathOperator{\Id}{Id}

\DeclareMathOperator{\GL}{GL}
\DeclareMathOperator{\Aut}{Aut}

\newcommand{\dbar}{\bar\partial}

\newcommand{\scGC}{\mathscr{G}^{\mathbb{C}}}

\newcommand{\rk}{\mathrm{rank}}

\def\Int{\mathrm{Int}}

\def\R{\mathbb{R}}
\def\Q{\mathbb{Q}}

\def\N{\mathbb{N}}

\def\C{\mathbb{C}}

\def\cA{\mathcal{A}}

\def\cF{\mathcal{F}}
\def\cE{\mathcal{E}}

\def\cG{\mathcal{G}}

\def\cS{\mathcal{S}}

\def\cK{\mathcal{K}}

\def\cZ{\mathcal{Z}}

\def\cS{\mathcal{S}}

\def\cZ{\mathcal{Z}}
\def\ocZ{\overline{\mathcal{Z}}}

\def\ua{\underline{a}}

\def\Gr{\mathrm{Gr}}

\def\delb{\overline{\partial}}

\def\rank{\mathrm{rank}}
\def\trace{\mathrm{trace}}
\def\Lie{\mathrm{Lie}}
\def\k{\mathfrak{k}}

\def\g{\mathfrak{g}}

\def\aut{\mathfrak{aut}}

\def\uep{\underline{\varepsilon}}

\def\Om{\Omega}
\def\om{\omega}
\def\ep{\varepsilon}
\def\>{\rangle}

\def\<{\langle}
\def\>{\rangle}

\title[Local wall-crossing for HYM connections]{Semi-stability and local wall-crossing for Hermitian Yang--Mills connections}

\author[A. Clarke]{Andrew Clarke}
\address{Andrew Clarke, Instituto de Matem\'atica, Universidade Federal do Rio de Janeiro, Av. Athos da Silveira Ramos 149, Rio de Janeiro, RJ, 21941-909, Brazil}
\email{andrew@im.ufrj.br} 

\author[C. Tipler]{Carl Tipler}
\address{Carl Tipler, Univ Brest, UMR CNRS 6205, Laboratoire de Math\'ematiques de Bretagne Atlantique, France}
\email{Carl.Tipler@univ-brest.fr}


\begin{document}

\begin{abstract}
We consider a sufficiently smooth semi-stable holomorphic vector bundle over a compact K\"ahler manifold. Assuming the automorphism group of its graded object to be abelian, we provide a semialgebraic decomposition of a neighbourhood of the polarisation in the K\"ahler cone into chambers characterising (in)stability. For a path in a stable chamber converging to the initial polarisation, we show that the associated HYM connections converge to an HYM connection on the graded object.
\end{abstract}

\maketitle 

\section{Introduction}
\label{sec:intro}
Slope stability, as introduced by Mumford for curves and generalized by Takemoto in higher dimension (\cite{Mum62,Tak72}) is a stability notion that can be used to construct and study moduli spaces of vector bundles. It depends on a choice of a polarisation, the variations of which being responsible for wall-crossing phenomena (see \cite{GrRoTo} for a survey on constructions and variations of moduli spaces of stable bundles). By the Hitchin--Kobayashi correspondence (\cite{skobayashi82,lubke83,uhlenbeckyau86,donaldson87}), a holomorphic vector bundle over a compact K\"ahler manifold is slope polystable if and only if it carries an Hermite--Einstein metric, or equivalently an hermitian Yang--Mills connection. While wall-crossing phenomena describe global variations of the moduli spaces from the algebraic point of view, our focus is on the local variations, from the analytic point of view, and inspired by \cite{DerMacSek}. While a stable bundle will carry an Hermite--Einstein metric with respect to any nearby polarisation, the situation for a semi-stable bundle is much more delicate, and this is the problem that we address in this paper. 

To state our results, let us first introduce some notations. Denote the K\"ahler cone of a compact K\"ahler manifold $X$ by $\cK_X\subset H^{1,1}(X,\R)$. Recall that a semi-stable holomorphic vector bundle $E\to (X,[\om])$ is said to be {\it sufficiently smooth} if its graded object $\Gr(E)$ is locally free. In that case, we define 
$$
\mathfrak{E}_{[\omega]}=\lbrace F\subset E \: \vert \: \mu_{[\om]}(F)=\mu_{[\om]}(E) \rbrace
$$
the set of subbundles of $E$ with same slope. Such bundles are all built out of successive extensions of the stable pieces of the graded object, and thus $\mathfrak{E}_{[\omega]}=\lbrace F_1,\ldots,F_p\rbrace$ is finite. For each $F_i\in \mathfrak{E}_{[\omega]}$, the map 
$$
\begin{array}{cccc}
 \nu_i : & \cK_X & \to & \R \\
         &  \alpha & \mapsto &\displaystyle \frac{c_1(E)\cdot \alpha^{n-1}}{\rank(E)} - \frac{c_1(F_i)\cdot \alpha^{n-1}}{\rank(F_i)}
\end{array}
$$
is polynomial in its arguments (seen as elements in the vector space $H^{1,1}(X,\R)$). We then set 
$$
\cS_s=\bigcap_{i=1}^p \lbrace \nu_i >0\rbrace , \: \cS_u=\bigcup_{i=1}^p \lbrace \nu_i<0 \rbrace \: \mathrm{ and } \: \cS_{ss}=\bigcap_{i=1}^p \lbrace \nu_i\geq 0 \rbrace,
$$
and for any ball $B([\om],R)\subset \cK_X$, we define $\cS_{s,R}=B([\om],R)\cap \cS_s$, and similarily for $\cS_{u,R}$ and $\cS_{ss,R}$. Those sets are semialgebraic by construction. 
\begin{proposition}
 \label{prop:intro}
 Let $E \to (X,[\om])$ be a simple and sufficiently smooth semi-stable holomorphic vector bundle over a compact K\"ahler manifold. Assume that the automorphism group $\Aut(\Gr(E))$ is abelian. Then, there is a ball $B([\om],R)\subset \cK_X$ in the K\"ahler cone and a semialgebraic partition $B([\om],R)=\cS_{s,R}\cup\cS_{u,R}\cup\cS_{ss,R}$, such that $E$ is stable (resp. unstable, strictly semi-stable) with respect to any $[\om']$ in $\cS_{s,R}$ (resp. in $\cS_{u,R}$, $\cS_{ss,R}$).
\end{proposition}
This proposition can be obtained easily from algebraic methods\footnote{A direct and simple proof relying on Greb and Toma's techniques (\cite[Proposition 6.5]{GrTo}) together with Grothendieck's Lemma (\cite[Lemma 1.7.9]{HuLe}) was kindly communicated to us by Julius Ross.} (compare with the much more general result \cite[Theorem 6.6]{GrTo}). Our proof however relies on gauge theoretical methods, and provides information on the behaviour of the associated HYM connections when the polarisation reaches the boundary of $\cS_{s,R}$. 
\begin{theorem}
 \label{theo:intro}
 In the setting of Proposition \ref{prop:intro}, for any path $([\om_t])_{t\in (0,1]}$ in a connected component of $\cS_{s,R}$ such that $\underset{t\to 0}{\lim}[\om_t]=[\om]$, there is a path  $(A_t)_{t\in (0,1]}$ of HYM connections on $E$  with respect to $(\om_t)_{t\in (0,1]}$ such that $\underset{t\to 0}{\lim} A_t=A_0$ is a HYM connection on $\Gr(E)\to (X,[\om])$.
\end{theorem}
In the previous statement, when referring to a HYM connection $A$ on $E$ (resp. on $\Gr(E)$), we implicitly assumed that $A^{0,1}$ was gauge equivalent to the holomorphic connection, or Dolbeault operator, of $E$ (resp. of $\Gr(E)$). Hence, the conclusion of Theorem \ref{theo:intro} can be rephrased as follows. If $\delb_{\Gr(E)}$ (resp. $\delb_E$) stands for the Dolbeault operator of $\Gr(E)$ (resp. of $E$) and if $h$ is an Hermite--Einstein metric on $\Gr(E)$ with respect to $\om$, then there are gauge transformations $(f_t)_{0<t\leq 1}$ such that the Chern connections of $(f_t^*\delb_E,h)$ are HYM with respect to $(\om_t)_{0<t\leq 1}$ and converge to the Chern connection of $(\delb_{\Gr(E)},h)$ in any Sobolev norm (see Section \ref{sec:HYM} for definitions of gauge transformations and their action). Our construction actually provides information on the HYM connections associated to {\it any} path in a connected component of $\cS_{s,R}$ converging to {\it any} class in $\cS_{ss,R}$, see Remark \ref{rem:concludingremark}.

Note that Theorem \ref{theo:intro} is not a direct application of the Hitchin--Kobayashi correspondence, nor of the implicit function theorem due to the presence of the symmetries $\Aut(\Gr(E))$ that obstruct the linear theory. Our proof instead make use of the moment map interpretation of the equations by Atiyah--Bott and Donaldson (\cite{AtBot,donaldson85}), together with a Kuranishi slice construction inspired by \cite{donaldson-largesym,gabor-deformations,BuSchu} and a local version of Atiyah and Guillemin--Sternberg's convexity theorem for images of moment maps.

\begin{remark}
 While the hypothesis on $\Gr(E)$ being locally free is natural, asking for its automorphism group to be abelian is purely technical. This is used in Section \ref{sec:proofmain}, which relies on a similar construction for cscK metrics \cite{SekTip22}. In \cite{SekTip}, similar results are obtained for perturbations of K\"ahler classes and semi-stable bundles. However, in \cite{SekTip}, an extra hypothesis on $\Gr(E)$ was required (unicity of the Jordan--Holder filtration with locally-free stable quotients), the proof was more technical, and the results only hold for perturbations along lines (of the form $t\mapsto [\om]+t\alpha$) in the K\"ahler cone.
\end{remark}
\begin{remark}
 The hypothesis on $E$ being simple can easily be dropped. In general, one might look for polystable perturbations of a semi-stable bundle. In that case, one should start from a direct sum of semi-stable bundles of the same slopes, and deal with each summands separatly. To ease the exposition, we will restrict to the simple case.
\end{remark}

It is not difficult to see that Theorem \ref{theo:intro} holds true for specific compact families of semi-stable holomorphic vector bundles that share $\Gr(E)$ as graded object (as in the similar problem for cscK metrics, see the discussion in \cite[Section 4.5]{SekTip22}). The next step towards the understanding of local wall-crossing phenomena for HYM connections is to obtain a uniform version of Theorem \ref{theo:intro} for all small deformations of $\Gr(E)$ at once.

\subsection*{Organisation of the paper:}  First, in Section \ref{sec:setup}, we recall the basics on HYM connections and slope stability. Then, in Section \ref{sec:perturbation}, we produce a family of Kuranishi slices parametrising small deformations of $\Gr(E)$, where each slice depends on a K\"ahler class near $[\om]$. This step is inspired by \cite{gabor-deformations,BuSchu,DerMacSek,SekTip22}, and enables to reduce the problem to the study of a family of finite dimensional moment maps. Section \ref{sec:proofmain} is then devoted to the proof of Theorem \ref{theo:intro}, relying on the recent technique from \cite{SekTip22} to control the perturbed moment maps.

\subsection*{Acknowledgments:} The authors benefited from visits to LMBA and Gotheborg University; they would like to thank these welcoming institutions for providing stimulating work environments. The idea of this project emerged from discussions with Lars Martin Sektnan, whom we thank for sharing his ideas and insight. We also thank Julius Ross for kindly answering our questions on the chamber structure, and pointing to us reference \cite{GrTo}. AC is partially supported by the grants {BRIDGES ANR--FAPESP ANR-21-CE40-0017} and { Projeto CAPES - PrInt UFRJ 88887.311615/2018-00}. CT is partially supported by the grants MARGE ANR-21-CE40-0011 and BRIDGES ANR--FAPESP ANR-21-CE40-0017.

\section{Preliminaries}
\label{sec:setup}
In Sections \ref{sec:HYM} and \ref{sec:slopestability} we introduce the notions of HYM connections and slope stability, together with some general results, and refer the reader to \cite{skobayashi87} and \cite{HuLe}. 

\subsection{The hermitian Yang--Mills equation}
\label{sec:HYM}

Let $E \to X$ be a holomorphic vector bundle over a compact K\"ahler manifold $X$. A hermitian metric on $E$ is \textit{Hermite--Einstein} with respect to a K\"ahler metric with K\"ahler form $\omega$ if the curvature $F_h \in \Omega^2 \left(X, \End E \right)$ of the corresponding Chern connection satisfies
\begin{align}\label{HEeqn} \Lambda_{\om} \left( iF_h \right) = c \Id_E
\end{align}
for some real constant $c$. Equivalently, if $h$ is some hermitian metric on the smooth complex vector bundle underlying $E$, a hermitian connection $A$ on $(E,h)$ is said to be \textit{hermitian Yang--Mills} if it satisfies
\begin{equation*}
\left\{
\begin{array}{ccc}
F_A^{0,2} & = & 0, \\
\Lambda_{\om} \left( i F_A \right) & = & c \Id_E.
\end{array}
\right.
\end{equation*}
The first equation of this system implies that the $(0,1)$-part of $A$ determines a holomorphic structure on $E$, while the second that $h$ is Hermite--Einstein for this holomorphic structure. We will try to find hermitian Yang--Mills connections within the complex gauge group orbit, which we now define. The \textit{complex gauge group} is 
$$ 
\scGC(E) = \Gamma \left( \GL \left(E, \mathbb{C} \right) \right),
$$
and its {\it hermitian} version is
$$ 
\scGC(E,h) = \Gamma \left( \GL \left(E, \mathbb{C} \right) \right)\cap \Gamma \left( \End_H (E, h) \right),
$$
where $\End_H (E,h)$ stands for the hermitian endomorphisms of $(E,h)$.
Note that if $\dbar$ is the Dolbeault operator defining the holomorphic structure on $E$, then $ f \circ \dbar \circ f^{-1}$ defines a biholomorphic complex structure on $E$. Let $d_A = \partial_A + \dbar_A$ be the Chern connection of $(E,h)$ with respect to the original complex structure (that is $\dbar_A = \dbar$). Then the Chern connection $A^f$ of $h$ with respect to $f \circ \dbar \circ f^{-1}$ is 
$$
d_{A^f} = (f^*)^{-1} \circ \partial_A \circ (f^*) + f \circ \dbar \circ f^{-1}.
$$
Solving the hermitian Yang--Mills equation is equivalent to solving $$\Psi (s) = c \Id_E$$
where 
$$
\begin{array}{cccc}
\Psi :  & \Lie(\scGC(E,h)) &  \longrightarrow  & \Lie(\scGC(E,h))\\
 & s & \longmapsto & i\Lambda_\om(F_{A^{\exp(s)}}),
\end{array}
$$
and where $\Lie(\scGC(E,h)):=i\Gamma(\End_H(E,h))$ is the tangent space to $\scGC(E,h)$ at the identity. For a connection $A$ on $E$, the Laplace operator $\Delta_{A}$ is
\begin{align}\label{laplaceop}  \Delta_A =  i \Lambda_{\omega} \left( \bar \partial_{A}  \partial_{A} -  \partial_{A} \bar \partial_{A} \right).
\end{align}
If $A_{\End E}$ denote the connection induced by $A$ on $\End E$, then :
\begin{lemma}\label{lem:linop} 
If $A$ is the Chern connection of $(E,\delb,h)$, the differential of $\Psi$  at identity is $$ d \Psi_{\Id_E} = \Delta_{A_{\End E}}.$$
 If moreover $A$ is assumed to be hermitian Yang--Mills, then the kernel of $ \Delta_{A_{\End E}}$ acting on $\Gamma(\End(E))$ is given by the Lie algebra $\aut(E)$ of the space of automorphisms $\Aut(E)$ of $(E,\delb)$.
\end{lemma}
The last statement about the kernel follows from the K\"ahler identities and the Akizuki-Nakano identity that imply $\Delta_{A_{\End E}}=\partial^*_{A}\partial_{A}+\bar\partial_{A}^*\bar\partial_{A}$, the two terms of which are equal if $A$ is Hermitian Yang-Mills. The operator $\Delta_{A_{\End E}}$ being elliptic and self-adjoint, $\aut(E)$ will then appear as a cokernel in the linear theory for perturbations of hermitian Yang--Mills connections.

\subsection{Slope stability}
\label{sec:slopestability}
We recall some basic facts about slope stability, as introduced by \cite{Mum62,Tak72}, and refer the interested reader to \cite{HuLe} for a detailed treatment. We denote here $L:=[\om]$ the polarisation of the $n$-dimensional K\"ahler manifold $X$.
\begin{definition}
\label{def:stability}
For $\cE$ a torsion-free coherent sheaf on $X$, the slope $\mu_L(\cE)\in\Q$ (with respect to $L$) is given by the intersection formula
 \begin{equation}
  \label{eq:slope}
  \mu_L(\cE)=\frac{\deg_L(\cE)}{\rank(\cE)},
 \end{equation}
 where $\rank(\cE)$ denotes the rank of $\cE$ while $\deg_L(\cE)=c_1(\cE)\cdot L^{n-1}$ stands for its degree. Then, $\mathcal{E}$ is said to be \emph{slope semi-stable} (resp. \emph{slope stable}) with respect to $L$ if for any coherent and saturated subsheaf $\mathcal{F}$ of $\mathcal{E}$ with $0<\rank( \cF)<\rank(\cE)$, one has
\begin{eqnarray*}
\mu_L(\mathcal{F})\leq \mu_L(\mathcal{E}) \:\textrm{( resp. } \mu_L(\mathcal{F})< \mu_L(\mathcal{E}) ).
\end{eqnarray*}
A direct sum of slope stable sheaves of the same slope
 is said to be \emph{slope polystable}.
\end{definition}
In this paper, we will often omit ``slope'' and simply refer to stability of a sheaf, the polarisation being implicit. We will make the standard identification of a holomorphic vector bundle $E$ with its sheaf of sections, and thus talk about slope stability notions for vector bundles as well. In that case slope stability relates nicely to differential geometry via the Hitchin--Kobayashi correspondence :
\begin{theorem}[\cite{skobayashi82,lubke83,uhlenbeckyau86,donaldson87}]
\label{thm:HKcorrespondence}
There exists a Hermite--Einstein metric on $E$ with respect to $\omega$ if and only if $E$ is polystable with respect to $L$
\end{theorem}
We will be mostly interested in semi-stable vector bundles. A \emph{Jordan--H\"older filtration} for a torsion-free sheaf $\cE$ is a filtration by coherent and saturated subsheaves:
 \begin{align}
\label{eq:JHfiltration}
0 = \cF_0 \subset \cF_1 \subset \hdots \subset \cF_\ell = \cE,
\end{align}
such that the corresponding quotients,
\begin{align}
\label{eq:JHquots}
\cG_i = \frac{\cF_i}{\cF_{i-1}},
\end{align}
 for $i=1, \hdots, \ell$, are stable with slope $\mu_L(\cG_i)=\mu_L(\cE)$.
 In particular, the graded object of this filtration
 \begin{equation}
  \label{eq:JHgraded}
  \Gr(\cE):=\bigoplus_{i=1}^l \cG_i
 \end{equation}
is polystable. 
 
\section{An adapted family of Kuranishi slices}
\label{sec:perturbation}
Let $(X,[\om])$ be a $n$-dimensional compact K\"ahler manifold, and let $([\alpha_i])_{1\leq i\leq p}$ be a basis of the real vector space $H^{1,1}(X,\C)\cap H^{2}(X,\R)$. For $\uep=(\ep_1,\ldots, \ep_p)\in\R^p$, we define 
$$
\om_{\uep}:=\om+\sum_{i=1}^p \ep_i \alpha_i \in\Om^{1,1}(X,\R).
$$
We fix a small neighbourhood $U$ of zero in $\R^p$ such that for any $\uep\in U$, $[\om_{\uep}]$ defines a K\"ahler class on $X$. When considering slopes, we will use the notation $L_{\uep}:=[\om_{\uep}]$.

Let $\Gr(E)=\bigoplus_{i=1}^\ell \cG_i$ be a polystable holomorphic vector bundle over $(X,[\om])$ with stable components $\cG_i$ (this will be the graded object of some semi-stable bundle later on). Thanks to the Hitchin--Kobayashi correspondence, we can fix an Hermite--Einstein metric $h_0$ on $\Gr(E)$. We denote by $\delb_0$ the holomorphic connection on $(\Gr(E),h_0)$ and by $A_0$ the associated (HYM) Chern connection. We will be interested in HYM connections on small deformations of $\Gr(E)$, for the perturbed polarisations $([\om_{\uep}])_{\uep\in U}$. The goal of this section is to reduce the problem of finding a zero for the operator $s\mapsto i\Lambda_{\om_{\uep}}(F_{A_0^{\exp(s)}})-c_\varepsilon\mathrm{Id}$ in a gauge group orbit to a finite dimensional problem. Note that we don't need the asumption on $\Aut(\Gr(E))$ being abelian in this section.

\subsection{Perturbing Kuranishi's slice}
\label{sec:perturbslice}
 The automorphism group $G:=\Aut(\Gr(E))$ is a reductive Lie group with Lie algebra $\g:=\aut(\Gr(E))$ and compact form $K\subset G$, with $\k:=\Lie(K)$. 
 
 Our starting point will be the following proposition, whose proof follows as in \cite{kuranishi} (see also \cite{BuSchu,doan} for a detailed treatment). We introduce the notation 
$$
V:=H^{0,1}(X,\End(\Gr(E)))
$$
for the space of harmonic $(0,1)$-forms with values in $\Gr(E)$, where the metrics used to compute adjoints are $\om$ on $X$ and $h_0$ on $\Gr(E)$. Note that the $G$-action on $\Gr(E)$ induces a linear representation $G \to \mathrm{GL}(V)$. 
\begin{proposition}
 \label{prop:firstslice}
 There exists a holomorphic $K$-equivariant map 
 $$
 \Phi : B \to \Om^{0,1}(X,\End(\Gr(E)))
 $$
 from a ball around the origin $B\subset V$ such that :
 \begin{enumerate}
  \item $\Phi(0)=0$;
  \item $Z:=\lbrace b\in B\,\vert\: (\delb_0+\Phi(b))^2=0\rbrace$ is a complex subspace of $B$;
  \item if $(b,b')\in Z^2$ lie in the same $G$-orbit, then $\delb_0+\Phi(b)$ and $\delb_0+\Phi(b')$ induce isomorphic holomorphic bundle structures;
  \item The $\scGC(\Gr(E))$-orbit of any small complex deformation of $\Gr(E)$ intersects $\Phi(Z)$.
 \end{enumerate}
\end{proposition}
 The space $Z$ corresponds to the space of integrable Dolbeault operators in the image of $\Phi$, and $\Phi(B)$ is a {\it slice} for the gauge group action on the set of Dolbeault operators nearby $\delb_0$. 
 
 The next step will be to perturb $\Phi$. The ideas here go back to \cite{donaldson-largesym,gabor-deformations}, and our framework will be that of \cite{BuSchu}. The strategy to do this in family with respect to the parameter $\uep$ was inspired by \cite{DerMacSek,SekTip22}.
 
 Given the metrics $\om$ and $h_0$, together with the covariant derivatives given by $A_0$, we can introduce $L^{2,l}$ Sobolev norms on spaces of sections. We will denote by $\cE^l$ the $L^{2,l}$ Sobolev completion of any space of sections $\cE$. In what follows, $l\in\N^*$ will be assumed large enough for elements in $\cE^l$ to admit as much regularity as required.
 
\begin{proposition}
 \label{prop:perturbingslice}
 Up to shrinking $U\times B$, there is a  continuously differentiable map
 $$
 \widetilde\Phi : U\times B \to \Om^{0,1}(X,\End(\Gr(E)))^l
 $$
 such that if $ A_{\uep,b}$ is the Chern connection of $(\delb_0+ \widetilde \Phi(\uep,b),h_0)$:
 \begin{enumerate}
  \item for all $(\uep,b)\in U\times Z$, $\delb_0+\widetilde \Phi(\uep,b)$ and $\delb_0+\Phi(b)$ are gauge equivalent,
  \item for all $\uep\in U$, the map  $b\mapsto\Lambda_{\om_{\uep}} i F_{ A_{\uep,b}}$ takes values in $\k$.
 \end{enumerate}
\end{proposition}
\begin{remark}
 By elliptic regularity, elements in the image of $\widetilde \Phi$ will actually be smooth. However, regularity of the map $\widetilde \Phi$ is with respect to the $L^{2,l}$ Sobolev norm. 
\end{remark}
To ease notations, we will use $\Lambda_{\uep}$ for the Lefschetz operator of $\om_{\uep}$. We also introduce the topological constants 
\begin{eqnarray*}
c_{\uep} = \frac{2\pi n}{\mathrm{vol}_{\omega_{\uep}}(X)}\frac{\left( c_1(\Gr(E))\cup [\omega_{\uep}]^{n-1}\right)[X]}{\rk(\Gr(E))}.
\end{eqnarray*}

\begin{proof}
 For $b\in B$, we will denote by $A_b$ the Chern connection of $(\delb_0+  \Phi(b),h_0)$. Note that in particular $A_0$ is a HYM connection on $\Gr(E)$. The aim is to apply the implicit function theorem to perturb $A_b$ along gauge orbits in order to satisfy point $(2)$ of the statement. For $s\in \Gamma(X,\End (\Gr(E)))$ we define
$A_b(s)=A_b^{\exp(s)}$.
  By the regularity of $\Phi$, the assignment $(b,s)\mapsto A_b(s)-A_0$  (resp. $(b,s)\mapsto F_{A_b(s)}$) is smooth from $B\times \Gamma(X,\End (\Gr(E)))^l$ to $\Om^1(X,\End(\Gr(E)))^{l-1}$ (resp. $\Om^2(X,\End(\Gr(E)))^{l-2}$). We deduce that the operator 
 $$
 \begin{array}{cccc}
 \Psi : & U\times B \times \Gamma(X,\End_H(\Gr(E),h_0))^l  & \to & \Gamma(X,\End_H(\Gr(E),h_0))^{l-2} \\
       &       (\uep, b,s) & \mapsto & \Lambda_{\uep} i F_{A_b(s)}-c_{\uep}\Id_{\Gr(E)}
 \end{array}
 $$
 is a $\mathcal{C}^1$ map. As $A_0$ is HYM, we have $\Psi(0)=0$. By Lemma \ref{lem:linop}, its differential in the $s$ direction at zero is given by the Laplace operator $\Delta_{A_0}$ of $A_0$, whose co-kernel is $i\k\subset\Gamma(X,\End_H(\Gr(E),h_0)) $. Then, by a standard projection argument onto some orthogonal complement of $i\k$,  we can apply the implicit function theorem and obtain a $\mathcal{C}^1$ map $(\uep,b)\mapsto s(\uep,b)$ such that $\Psi(b,\uep,s(\uep,b))$ lies in $\k$, and conclude the proof by setting 
$$\widetilde \Phi(\uep,b)=(A_b(\uep,s(\uep,b)))^{0,1}-A_0^{0,1}.
$$
\end{proof}
\subsection{The finite dimensional moment maps}
\label{sec:finitedimmmap}
We will now explain that for each $\uep\in U$, the map 
\begin{equation}
 \label{eq:finitedimmoment}
\begin{array}{cccc}
 \mu_{\uep} : & B & \to & \k \\
           & b &  \mapsto & \Lambda_{\uep} i F_{ A_{\uep,b}} - c_{\uep} \Id_{\Gr(E)}
\end{array}
\end{equation}
is a moment map for the $K$-action on $B$, for suitable symplectic forms $\Om_{\uep}$ on $B$. Recall from \cite{AtBot,donaldson85} that for $\uep\in U$, the gauge action of $\scGC(\Gr(E),h_0)$ on the affine space $\delb_0+\Om^{0,1}(X,\End(\Gr(E)))$ is hamiltonian for the symplectic form given, for $(a,b)\in \Om^{0,1}(X,\End(\Gr(E)))^2$, by 
\begin{equation}
 \label{eq:momentmapdonaldson}
\Om^{D}_{\uep}(a,b)=\int_{X} \trace(a\wedge b^*)\wedge \frac{\om_{\uep}^{n-1}}{(n-1)!},
\end{equation}
with equivariant moment map $\delb \mapsto \Lambda_{\uep} F_{A_{\delb}}$ where $A_{\delb}$ stands for the Chern connection of $(\delb,h_0)$. Here, we identified the Lie algebra of $\scGC(\Gr(E),h_0)$ with its dual by mean of the invariant pairing 
\begin{equation}
 \label{eq:pairing}
\langle s_1,s_2\rangle_{\uep} := \int_{X} \trace(s_1\cdot s_2^*)\; \frac{\om_{\uep}^n}{n!}.
\end{equation}
\begin{remark}
 We used above the Chern correspondence, for $h_0$ fixed, between Dolbeault operators and hermitian connections to express the infinite dimensional moment map picture on the space of Dolbeault operators.
\end{remark}
\begin{proposition}
 \label{prop:symplecticembedding}
 Up to shrinking $U\times B$, for all $\uep\in U$, the map $\delb_0+\widetilde \Phi(\ep,\cdot)$ is a $K$-equivariant map from $B$ to $\delb_0+\Om^{0,1}(X,\End(\Gr(E)))$ whose image is a symplectic submanifold for $\Om^D_{\uep}$.
\end{proposition}

\begin{proof}
The equivariance follows easily from Proposition \ref{prop:firstslice} and from the construction of $\widetilde \Phi$ in the proof of Proposition \ref{prop:perturbingslice}. The map $\widetilde \Phi $ is obtained by perturbing $\Phi$. But $\Phi$ is complex analytic with, by construction, injective differential at the origin (see e.g. the orginal proof \cite{kuranishi} or \cite{doan}).  Thus $\Phi(B)$ is a complex subspace of $\Om^{0,1}(X,\End(\Gr(E)))$. We deduce that, up to shrinking $B$, $\Phi$ induces an embedding of $B$ such that the restriction of $\Om^D_0$ to $\Phi(B)$ is non-degenerate (recall that $\Om^D_0$ is a K\"ahler form on the space of Dolbeault operators on $X$). As $\widetilde \Phi(\uep,\cdot)$ is obtained by a small and continuous perturbation of $\Phi$, and as being a symplectic embedding is and open condition, the result follows.
\end{proof}
From this result, we deduce that the map $\mu_{\uep}$ defined in (\ref{eq:finitedimmoment}) is a moment map for the $K$-action on $B$ with respect to the pulled back symplectic form 
$$\Om_{\uep}:=\widetilde \Phi(\ep,\cdot)^*\Om^D_{\uep},
$$
and where we use the pairing $\langle \cdot,\cdot \rangle_{\uep}$ defined in (\ref{eq:pairing}) to identify $\k$ with its dual.

We now assume that $\Gr(E)$ is the graded object of a simple, semi-stable and sufficiently smooth holomorphic vector bundle $E$ on $(X,[\om])$. As $E$ is built out of successive extensions of the stable components $\cG_i$'s of $\Gr(E)$, the Dolbeault operator $\delb_E$ on $E$ is given by
$$
\delb_E=\delb_0+\gamma
$$
where $\gamma \in\Om^{0,1}(X,\Gr(E)^*\otimes \Gr(E))$ can be written
$$
\gamma=\sum_{i<j} \gamma_{ij}
$$
for (possibly vanishing) $\gamma_{ij} \in\Om^{0,1}(X,\cG_j^* \otimes \cG_i)$. Elements 
$$g:=g_1\Id_{\cG_1}+\ldots,+g_\ell \Id_{\cG_\ell}\in G,$$
for $(g_i)\in(\C^*)^\ell$, act on $\delb_E$ and produce isomorphic holomorphic vector bundles in the following way :
\begin{equation}
\label{eq:explicitaction}
g\cdot \delb_E= \delb_0 + \sum_{i<j} g_ig_j^{-1}\gamma_{ij}.
\end{equation}
In particular, for $g=(t^\ell,t^{\ell-1},\ldots,t)$, letting $t\mapsto 0$, we can see $E$ as a small complex deformation of $\Gr(E)$.
By Proposition \ref{prop:firstslice}, the holomorphic connection $\delb_E$ is gauge equivalent to an element $\delb_b:=\delb_0+ \Phi(b)$ for some $b\in B$. Then, from properties of the maps $\Phi$ and $\widetilde \Phi$, for all $\uep\in U$ and for all $g\in G$, $\delb_E$ will be gauge equivalent to $\delb_0+\widetilde \Phi(\uep,g\cdot b)$, provided $g\cdot b \in B$. As a zero of $\mu_{\uep}$ corresponds to a HYM connection on $(X,\om_{\uep})$, we are left with the problem of characterising the existence of a zero for $\mu_{\uep}$ in the $G$-orbit of $b$.

\section{Proof of the main results}
\label{sec:proofmain}
We carry on with notations from the last section, and our goal now is to prove Proposition \ref{prop:intro} and Theorem \ref{theo:intro}. This is where we will need to assume that $G=\Aut(\Gr(E))$ is abelian. As $\Gr(E)=\bigoplus_{i=1}^\ell \cG_i$ this is equivalent to the fact that all stable components $\cG_i$ are non isomorphic, and this gives the explicit description 
$$
\g=\aut(\Gr(E))=\bigoplus_{i=1}^\ell \C\cdot\Id_{\cG_i}.
$$
Up to shrinking $U$, we will assume that the set $\lbrace [\om_{\uep}],\; \uep\in U\rbrace$  is a ball $B([\om],R)\subset \cK_X$. By definition of the set $\cS_u$ in  Section \ref{sec:intro}, it is clear that for any $[\om']\in \cS_{u,R}$, $E$ will be unstable with respect to $[\om']$. In the sequel, we will first show that, up to shrinking $R$, $E$ admits HYM connections with respect to all classes in $\cS_{s,R}$. Together with simplicity of $E$, this will imply that $E$ is stable for any class in $\cS_{s,R}$. Then, we will explore the convergence properties of the constructed HYM connections to conclude proof of Theorem \ref{prop:intro}. Finally, we will settle the semi-stable case, for classes in $\cS_{ss,R}$, and the proof of Proposition \ref{prop:intro} will be complete.

\subsection{The local convex cone associated to the $K$-action}
\label{sec:reducing}
Our first goal is to show that for any $\uep\in U$ such that $[\om_{\uep}]\in \cS_{s,R}$, 
there is a zero of $\mu_{\uep}$ in $$\cZ:=G\cdot b\cap B.$$
In order to do so, we start by describing, at least locally, the images of $\cZ$ by the maps $(\mu_{\uep})_{\uep\in U}$. In this section, relying on \cite{SekTip22}, we will see that those images all contain translations of (a neighbourhood of the apex of) the same convex cone.

By simplicity of $E$, the stabiliser of $b$ under the $K$-action is reduced to the $S^1$-action induced by gauge transformations of the form $e^{i\theta}\Id_E$. As those elements fix all the points in $B$, elements in $S^1\cdot \Id_E$ will play no role in the arguments that follow. Hence, we will work instead with the quotient torus $K_0:=K/S^1\cdot \Id_E$. Note that the constants $c_{\uep}$ that appear in the maps $\mu_{\ep}$ in (\ref{eq:finitedimmoment}) are chosen so that $\langle \mu_{\uep}, \Id_E \rangle_{\uep}=0$. As the $\mu_{\uep}$ take values in $\k$, this is equivalent to say $\trace(\mu_{\uep})=0$. Hence, setting $\k_0\subset \k$ to be the set of trace free elements in $\bigoplus_{i=1}^\ell i\R \cdot \Id_{\cG_i}$, we will consider the family of moment maps $\mu_{\uep}:B\to\k_0$ for the $K_0$-action, and we may, and will, assume that the stabiliser of $b$ is trivial. Using the inner product $\langle \cdot, \cdot \rangle_{\uep}$ to identify $\k_0\simeq\k_0^*$, we can see the maps $\mu_{\uep}$ as taking values in $\k_0^*$ :
 $$
 \mu_{\ep}^* : B \to \k_0^*.
 $$
  There is a weight decomposition of $V$ under the abelian $K$-action
 \begin{equation}
  \label{eq:weightdecomposition}
 V:= \bigoplus_{m\in M} V_m
 \end{equation}
 for $M\subset \k_0^*$ the lattice of characters of $K_0$. In the matrix blocks decomposition of $V=H^{0,1}(X,\End(\Gr(E)))$ induced by $\Gr(E)=\bigoplus_{i=1}^\ell\cG_i$, using the product hermitian metric $h_0$, we have 
 $$
 V=\bigoplus_{1\leq i,j\leq \ell} H^{0,1}(X,\cG_i^*\otimes\cG_j).
 $$
 The action of $g=(g_1,\ldots,g_\ell)\in K_0\simeq (\C^*)^\ell$ on $\gamma_{ij}\in V_{ij}:=H^{0,1}(X,\cG_i^*\otimes\cG_j)$ is given by:
 \begin{equation}
 \label{eq:action2}
 g\cdot \gamma_{ij}=g_ig_j^{-1} \gamma_{ij}.
 \end{equation}
Thus, in the weight space decomposition (\ref{eq:weightdecomposition}), $V_{ij}$ is the eigenspace with weight 
 \begin{equation}
  \label{eq:weightsmij}
 m_{ij}:= (0,\ldots,0,1,0,\ldots, 0, -1,0,\ldots,0)
 \end{equation}
 where $+1$ appears in $i$-th position and $-1$ in the $j$-th position. If we decompose $b$ accordingly as
 \begin{equation}
  \label{eq:indicicesdecompositionb}
 b=\sum_{ij} b_{ij},
 \end{equation}
 where $b_{ij}\in V_{ij}$ is non zero, as $\delb_E=\delb_0 +\gamma$ with $\gamma$ upper triangular, or equivalently as $E$ is obtained as successive extensions of the stable components $\cG_i$'s, only indices $(i,j)$ with $i < j$ will appear in (\ref{eq:indicicesdecompositionb}).
From now on, we will restrict our setting to
 $$
 B\cap \bigoplus_{b_{ij}\neq 0} V_{ij},
 $$
which we still denote by $B$. That is, we only consider weight spaces that appear in the decomposition of $b$. Similarily, we use the notation $V$ for $\bigoplus_{b_{ij}\neq 0} V_{ij}$.

To sum up, we are in the following setting :
\begin{enumerate}
 \item[($R_1$)] The compact torus $K_0$ acts effectively and holomorphically on the complex vector space $V$;
 \item[($R_2$)] There is a continous family of symplectic forms $(\Om_{\uep})_{\uep\in U}$ on $B\subset V$ around the origin, with respect to which the $K_0$-action is hamiltonian;
 \item[($R_3$)] The point $b\in B$ has trivial stabiliser, $0$ in its $K_0^\C$-orbit closure, and for all weight $m_{ij}\in M$ appearing in the weight space decomposition of $V$, $b_{ij}\neq 0$.
 \item[($R_4$)] The restriction of the symplectic form $\Om_0$ to the $K_0^\C$-orbit of $b$ is non-degenerate.
\end{enumerate}
This last point follows as in the proof of Proposition \ref{prop:symplecticembedding}.
We set $$\ocZ:=B\cap(\overline{K_0^\C\cdot b}).$$
We also introduce 
$$
\sigma:= \sum_{b_{ij}\neq 0} \R_+ \cdot m_{ij}\subset \k_0^*
$$
with $\lbrace m_{ij},\: b_{ij}\neq 0 \rbrace$ the set of weights that appear in the decomposition of $b\in V$,
and for $\eta >0$ 
$$
\sigma_\eta := \sum_{b_{ij}\neq 0} [0,\eta ) \cdot m_{i j}\subset \k_0^*.
$$
Note that by the local version of Atiyah and Guillemin--Sternberg's convexity theorem, there exists $\eta >0$ such that $\mu_{\uep}^*(0)+\sigma_\eta\subset\mu_{\uep}^*(B)$ for all $\uep$ small enough (see the equivariant Darboux Theorem \cite[Theorem 3.2]{Dwivedi} combined with the local description of linear hamiltonian torus actions \cite[Section 7.1]{Dwivedi}). By \cite[Proposition 4.6]{SekTip22}, the properties $(R_1)-(R_4)$ listed above actually imply :
 \begin{proposition}
 \label{prop:sigma-eta-image-orbit}
  Up to shrinking $U\times B$, there exists $\eta>0$ such that for all $\uep\in U$,
  $$
   \mu_{\uep}^*(0)+ \Int(\sigma_\eta)  \subset \mu^*_{\uep}(\cZ)
  $$
  and 
  $$
  \mu_{\uep}^*(0)+ \sigma_\eta  \subset \mu^*_{\uep}(\ocZ).
  $$
 \end{proposition}
 \begin{remark}
 The fact that the interior of $ \mu_{\uep}^*(0)+ \sigma_\eta$ is included in the image of the $K_0^\C$-orbit of $b$ by $\mu_{\uep}^*$ is not stated explicitely in \cite{SekTip22}, but follows from the discussion at the beginning of the proof of \cite[Proposition 4.6]{SekTip22}.
 \end{remark}

\subsection{Finding zeros of $\mu_{\uep}$ when $[\om_{\uep}]\in\cS_{s,R}$}
\label{sec:solving}
We now assume that $[\om_{\uep}]\in\cS_{s,R}$. Note that from the definition of $\cS_{s,R}$, this is equivalent to the fact that for any $\cF\in\mathfrak{E}_{[\omega]}$, 
$$
\mu_{L_{\uep}}(\cF)<\mu_{L_{\uep}}(E),
$$
where we recall $L_{\uep}=[\om_{\uep}]$.
From Proposition \ref{prop:sigma-eta-image-orbit}, to find a zero of $\mu_{\uep}$ in $\cZ$ it is enough to show $-\mu_{\uep}^*(0)\in \Int(\sigma_\eta)$, which reduces to $-\mu_{\uep}^*(0)\in \Int(\sigma)$ for small enough $\uep$. Arguing as in \cite[Lemma 4.8]{SekTip22}, $\sigma$ and its dual $$\sigma^\vee:=\lbrace v\in \k_0\:\vert\: \langle m, v \rangle\geq 0\; \forall m\in\sigma \rbrace$$ are strongly convex rational polyhedral cones of dimension $\ell-1$. Note that here the pairing $\langle \cdot,\cdot \rangle$ is the natural duality pairing. By duality, $\sigma=(\sigma^\vee)^\vee$, and we are left with the condition
$$
-\mu_{\uep}^*(0)\in \Int((\sigma^\vee)^\vee).
$$
The cone $\sigma^\vee$ can be written 
$$
\sigma^\vee=\sum_{\underline{a}\in\cA} \R_+\cdot v_{\underline{a}} 
$$
for a finite set of generators $\lbrace v_{\underline{a}} \rbrace_{\underline{a}\in\cA}\subset \k_0$. Hence, our goal now is to show that for all $\ua\in\cA$, $\langle \mu_{\uep}^*(0), v_{\ua} \rangle < 0$, which by construction is equivalent to
\begin{equation}
 \label{eq:toprove}
\langle \mu_{\uep}(0), v_{\ua} \rangle_{\uep}< 0,
\end{equation}
under the assumption that for any $\cF\in\mathfrak{E}_{[\omega]}$, 
\begin{equation}
\label{eq:wehave} 
\mu_{L_{\uep}}(\cF)<\mu_{L_{\uep}}(E).
\end{equation}
We will then study in more detail Equations (\ref{eq:toprove}) and (\ref{eq:wehave}). In order to simplify  notation, in what follows, we will assume that all the stable components of $\Gr(E)$ have rank one, so that $\trace(\Id_{\cG_i})=1$ for $1\leq i \leq \ell$. The general case can easily be adapted, and is left to the reader.

\subsubsection{Condition (\ref{eq:toprove}) : generators of the dual cone}
We will give here a more precise form for the generators $\lbrace v_{\ua}\rbrace_{\ua\in\cA}$ of $\sigma^\vee$. Recall from \cite[Section 1.2]{fulton} the method to find such generators : as $\sigma$ is $\ell-1$-dimensional, each of its facets is generated by $\ell-2$ elements amongst its generators $(m_{ij})$. Then, a generator $v_{\ua}$ for $\sigma^\vee$ will be an ``inward pointing normal'' to such a facet. Hence, if $$v_{\ua}=\sum_{i=1}^\ell a_i \Id_{\cG_i}$$ is a generator of $\sigma^\vee$, there exists a set $\cS:=\lbrace m_{ij} \rbrace$ of $\ell-2$ generators of $\sigma$  such that 
$$
\forall\; m_{ij}\in\cS,\: \langle m_{ij}, v_{\ua} \rangle=0.
$$
Moreover, $v_{\ua}\in \k_0$ should be trace free, and as we assume here $\rank(\cG_i)=1$ for all stable components, it gives 
$$
\sum_{i=1}^\ell a_i =0.
$$
\begin{lemma}
 \label{lem:dualgenerators}
 Up to scaling $v_{\ua}$, there exists a partition $\lbrace 1,\ldots,\ell \rbrace =I^-\cup I^+$ such that for all $i\in I^-$, $a_i=-\frac{1}{\sharp I^-}$ and for all $i\in I^+$, $a_i=\frac{1}{\sharp I^+}$, where $\sharp$ stands for the cardinal of a set.
\end{lemma}
\begin{proof}
 The key is to observe that if $m_{ij},m_{jk}\in \cS^2$, then $m_{ik}\notin \cS$. Indeed, by (\ref{eq:weightsmij}), $m_{ij}+m_{jk}=m_{ik}$, and those are generators of the cone. Equivalently, if $m_{ij},m_{ik}\in\cS^2$, then $m_{jk}\notin \cS$. We then assign an oriented graph $G_{\ua}$ to $v_{\ua}$. The vertices are labelled $a_1$ to $a_\ell$, and we draw an oriented edge from $a_i$ to $a_j$ if $a_i=a_j$ and $i<j$. For each $m_{ij}\in\cS$, $\langle m_{ij}, v_{\ua} \rangle =0$ gives $a_i=a_j$. Hence, $G_{\ua}$ has at least $\ell-2$ edges. To prove the result, it is enough to show that $G_{\ua}$ has $2$ connected components. Indeed, we can then set $I^-=\lbrace i\,\vert a_i <0 \rbrace$ and $I^+=\lbrace i\,\vert a_i >0 \rbrace$. All elements $a_i$ for $i\in I^-$ will correspond to the same connected component and be equal, and similarily with $i\in I^+$. As $\sum_{i=1}^\ell a_i =0$, we obtain the result by rescaling.
 
 Proving that $G_{\ua}$ has two connected components is then routine. It has $\ell$ vertices and $\ell-2$ oriented edges, with the rule that if there is an edge from $a_i$ to $a_j$ and an edge from $a_i$ to $a_k$, then there is no edge from $a_j$ to $a_k$. We consider the number of edges that start from $a_1$. If there are $\ell-2$ of those, then the connected component of $a_1$ has at least $\ell-1$ vertices, and we are left with at most $1$ singleton for the other component. The fact that $v_{\ua}$ is trace free imposes that there are at least $2$ connected components, and we are done in that case. Then, if there are $\ell-2-k$ edges from $a_1$, its connected component has at least $\ell-1-k$ elements, and we are left with at most $k+1$ vertices and $k$ edges for the other components. But its easy to show, by induction on $k$, that the rule stated above implies that there will be at most $1$ connected component for such a graph with $k+1$ vertices and $k$ edges, and we are done.
\end{proof}
We can now translate condition (\ref{eq:toprove}). By Lemma \ref{lem:dualgenerators}, it is equivalent to 
\begin{equation}
 \label{eq:toprove2}
 \frac{\sum_{i\in I^+}\langle \mu_{\uep}(0),\Id_{\cG_i}\rangle_{\uep}}{\sharp I^+}  < \frac{\sum_{i\in I^-}\langle \mu_{\uep}(0),\Id_{\cG_i}\rangle_{\uep}}{\sharp I^-}.
\end{equation}

\subsubsection{Condition (\ref{eq:wehave}) : one parameter degenerations}
We will associate to each generator $v_{\ua}$ of $\sigma^\vee$ a subsheaf $\cF\in\mathfrak{E}_{[\omega]}$. Geometrically, the idea is that $v_{\ua}\in\k_0$ generates a one-parameter subgroup of $K_0$ and a degeneration of $E$ to $\cF\oplus E/\cF$, to which is assigned the Hilbert--Mumford weight $\mu_{L_{\uep}}(\cF)-\mu_{L_{\uep}}(E)<0$. We let $v_{\ua}=\sum_{i=1}^{\ell} a_i \Id_{\cG_i}\in\sigma^\vee$ a generator as above, and define
$$
\cF_{\ua}=\bigoplus_{i\in I^+} \cG_i,
$$
as a {\it smooth} complex vector bundle, and will show that $\delb_E (\cF_{\ua})\subset \Om^{0,1}(X,\cF_{\ua})$. This implies that $\cF_{\ua}\in \mathfrak{E}_{[\omega]}$ as a {\it holomorphic} vector bundle, with Dolbeault operator the restriction of $\delb_E$. Recall that $\delb_E=\delb_0 +\gamma=\delb_0+ \sum_{b_{ij}\neq 0} \gamma_{ij}$, that is, by choice of $b$, the weights that appear in the weight decomposition of $\gamma$ are the same as those that appear in the decomposition of $b$. In the matrix blocks decomposition given by $\bigoplus_{i=1}^\ell\cG_i$, the operator $\delb_0$ is diagonal, and thus sends $\cF_{\ua}$ to $\Om^{0,1}(X,\cF_{\ua})$. We need to show that for each $j\in I^+$, $\gamma (\cG_j)\subset \Om^{0,1}(X,\cF_{\ua})$. As $v_{\ua}\in\sigma^\vee$, it satisfies, for any generator $m_{ij}$ of $\sigma$ :
$$
\langle m_{ij}, v_{\ua} \rangle \geq 0,
$$
that is, for all $(i,j)$ with $i<j$ and $b_{ij}\neq 0$,
$$
a_i - a_j \geq 0.
$$
As $j\in I^+$, this implies $a_i\geq a_j > 0$. Hence, if $i < j$ is such that $b_{ij}\neq 0$, then $i\in I^+$. Equivalently, for $i<j$, $i\in I^-$ implies $\gamma_{ij}=0$, and thus we see that $\gamma(\cG_j)\subset \Om^{0,1}(X,\cF_{\ua})$, and hence $\delb_E (\cF_{\ua})\subset \Om^{0,1}(X,\cF_{\ua})$. 

Then we have $\cF_{\ua}\in \mathfrak{E}_{[\omega]}$ and Condition (\ref{eq:wehave}) gives
$$
\mu_{L_{\uep}}(\cF_{\ua}) < \mu_{L_{\uep}}(E),
$$
which, by the see-saw property of slopes (see e.g. \cite[Corollary 3.5]{SekTip} ), is equivalent to 
$$
\mu_{L_{\uep}}(\cF_{\ua}) < \mu_{L_{\uep}}(E/\cF_{\ua})
$$
and thus (recall we assume $\rank(\cG_i)=1$):
\begin{equation}
\label{eq:wehave2}
\frac{\sum_{i\in I^+}\mu_{L_{\uep}}(\cG_i)}{\sharp I^+}  <\frac{\sum_{i\in I^-}\mu_{L_{\uep}}(\cG_i)}{\sharp I^-} .
\end{equation}

\subsection{Conclusion}
 By Chern--Weil theory, using the fact that $A_0$ and $ A_{\ep,0}$ are gauge-equivalent by point $(2)$ of Proposition \ref{prop:perturbingslice}, we have 
$$
\begin{array}{ccc}
\mu_{L_{\uep}}(\cG_i) & = & c_1(\cG_i)\cdot [\om_{\uep}]^{n-1} \\
               & = & \displaystyle\frac{1}{2\pi}\langle \mu_{\uep}(0), \Id_{\cG_i} \rangle_{\uep} + \frac{c_{\uep}}{2\pi}\langle \Id_E, \Id_{\cG_i} \rangle_{\uep}.
\end{array}
$$
We also have for all $i$'s, 
$$\langle \Id_E, \Id_{\cG_i} \rangle_{\uep}=\frac{[\om_{\uep}]^n}{n!}.$$
Inequality (\ref{eq:wehave2}) implies Inequality (\ref{eq:toprove2}), which concludes the existence of $b_{\uep}\in \cZ$ such that $\mu_{\uep}(b_{\uep})=0$. Then, by construction, the associated connections $ A_{\uep,b_{\uep}}$ provide HYM connections with respect to $\om_{\uep}$ on bundles gauge equivalent to $E$, where the gauge equivalences are given by elements in the finite dimensional Lie group $\Aut(\Gr(E))$. To show the convergence property of the connections as stated in Theorem \ref{theo:intro}, consider a path $t\mapsto [\om_t]$ in a connected component of $\cS_{s,R}$, converging to $[\om]$ when $t\mapsto 0$. This corresponds to a path $t\mapsto \uep(t)\in U$. It is then enough to show that the connections $ A_{\uep(t),b_{\uep(t)}}$ converge to $A_0= A_{0,0}$ in any $L^{2,l}$ Sobolev norm. By construction of $A_{\uep,b}$ in Proposition \ref{prop:perturbingslice}, it is enough to prove that $b_{\uep(t)}$ converges to $0$ when $\uep(t) \to 0$. Recall from \cite[Theorem 3.2 and Section 7.1]{Dwivedi} that $B$ can be chosen so that $\mu_{\uep}^*$ is given by
\begin{equation}
 \label{eq:expressionmuep}
\mu_{\uep}^*(b')=\mu_{\uep}^*(0)+\sum_{ij} \vert\vert b_{ij}'\vert\vert^2_{\uep} \cdot m_{ij},
\end{equation}
for some norm $\vert\vert \cdot \vert\vert_{\uep}$ that depends continously on $\uep$. As $\mu_{\uep}(0)\underset{\uep\to 0}{\to}\mu_0(0)=0$, the equation $\mu_{\uep}^*(b_{\uep})=0$ implies that for all $(i,j)$, $ \vert\vert (b_{\uep})_{ij}\vert\vert_{\uep} \underset{\uep\to 0}{\to} 0$. As the norms $\vert\vert\cdot\vert\vert_{\uep}$ vary continuously, they are mutually bounded, and thus $b_{\uep(t)}\underset{t\to 0}{\to}0$, which concludes the proof of the convergence property in Theorem \ref{theo:intro}. What remains is the semi-stable case. We need to show that if $[\om_{\uep}]\in\cS_{ss,R}$, $E$ is semi-stable for $[\om_{\uep}]$. The only remaining case to study is when for all $\cF\in\mathfrak{E}_{[\omega]}$,  $\mu_{L_{\uep}}(\cF)\leq \mu_{L_{\uep}}(E)$, with at least one equality. In that situation, the discussion in the last two sections shows that $-\mu_{\uep}(0)\in \sigma$ will lie in the boundary of $\sigma$. Hence, by Proposition \ref{prop:sigma-eta-image-orbit}, there is a boundary point $b'\in \ocZ$ in the orbit closure of $b$ with $\mu_{\uep}(b')=0$. This point corresponds to a HYM connection on a vector bundle that is then polystable for the holomorphic structure given by $ A_{\uep,b'}^{0,1}$, with respect to $L_{\uep}$. As this bundle corresponds to a boundary point in the complex orbit of $b$, it admits a small complex deformation to $E$. As semi-stability is an open condition, we deduce that $E$ is itself semi-stable for $L_{\uep}$.

\begin{remark}
 \label{rem:concludingremark}
 A little adaptation of the previous arguments leads to the following result. Assume that  $t\mapsto [\om_{\uep(t)}]$ is a path in a connected component of $\cS_{s,R}$ that converges to $[\om_{\uep(0)}]\in \cS_{ss,R}$. Then, we can find a filtration
 $$
 0\subset \cF_1\subset \ldots\subset \cF_{l}=E
 $$
 of $E$ by subbundles $\cF\in \mathfrak{E}_{[\omega]}$ such that for $i\in\lbrace 1, \ldots, l\rbrace$, 
 \begin{equation}
  \label{eq:concludingremark}
 \mu_{L_{\uep(0)}}(\cF_i)=\mu_{L_{\uep(0)}}(E)
 \end{equation}
 and $\rank(\cF_{i-1})$ is maximal amongst the ranks of elements $\cF\in\mathfrak{E}_{[\omega]}$ satisfying (\ref{eq:concludingremark}) and $\cF\subset \cF_i$. We then set 
 $$ 
 \cG:=\bigoplus_{i=1}^l \cF_i/\cF_{i-1}.
 $$
 Then, there is a path  $(A_t)_{t\in (0,1]}$ of HYM connections on $E$  with respect to the K\"ahler metrics $(\om_{\uep(t)})_{t\in (0,1]}$ such that $\underset{t\to 0}{\lim} A_t=A_0$ is a HYM connection on $\cG\to (X,[\om_{\uep(0)}])$, that is, the holomorphic connection of $A_0$ is equivalent to the one on $\cG$.
\end{remark}

\bibliographystyle{plain}	
 \bibliography{ClaSekTip}

\end{document}